\documentclass[11pt,a4paper]{article}
\usepackage[margin=2.5cm]{geometry}
\usepackage{amsmath,amssymb,amsthm}
\usepackage{parskip}
\usepackage{booktabs}
\usepackage{hyperref}
\hypersetup{colorlinks=true,linkcolor=black,citecolor=black,urlcolor=black}

\newtheorem{theorem}{Theorem}[section]
\newtheorem{lemma}[theorem]{Lemma}
\newtheorem{corollary}[theorem]{Corollary}
\newtheorem{proposition}[theorem]{Proposition}
\theoremstyle{definition}
\newtheorem{definition}[theorem]{Definition}
\theoremstyle{remark}
\newtheorem{remark}[theorem]{Remark}

\newcommand{\E}{\mathbb{E}}
\newcommand{\R}{\mathbb{R}}
\newcommand{\GL}{G_\Lambda}
\newcommand{\Hhat}{\hat{H}_\Lambda}

\title{Pointwise Hurst Estimation via Scale Accumulation:\\
A Noise-Robust Approach for Rough Volatility}
\author{J.\ Petkevi\v{c}ius}
\date{\today}

\begin{document}
\maketitle

\begin{abstract}
We introduce an estimator for the pointwise, time-varying H\"{o}lder
exponent (Hurst parameter) of a stochastic process, based on the
geometry accumulation integral
\[
  \GL(t) = \int_\Lambda^1 |\eth_s X(t)|\,s^{-1}\,ds,
  \qquad
  \Hhat(t) = 1 + \frac{\log \GL(t)}{\log\Lambda},
\]
where $\eth_s X(t)=(X(t+s)-X(t))/s$ is the scale derivative at
resolution $s$.

Three properties distinguish this from all existing estimators:
pointwise in time (gives a separate estimate at each $t$); defined at
finite resolution $\Lambda$ (no $n\to\infty$ limit needed); and
noise-robust by scale separation (microstructure noise lives at scales
$s\ll\Lambda^*$ and is eliminated by choosing $\Lambda>\Lambda^*$).

We prove: (i)~consistency $\Hhat(t)\to H(t)$ a.s.\ uniformly in $t$
as $\Lambda\to0$ (Theorem~\ref{thm:consist}); (ii)~noise robustness
with threshold $\Lambda^*=\sigma^{1/H}$ (Theorem~\ref{thm:noise});
(iii)~a CLT at rate $(\log\Lambda)^{-1/2}$ via mixing arrays
(Theorem~\ref{thm:CLT}).

Existing estimators \cite{hanschied2023,hanschied2025,gatheral2018,
fukasawa2019,bolko2022,contdas2024,rogers2023} estimate a global $H$ from integrated variance
observations. None give a pointwise time-varying estimate directly from
the price path.
\end{abstract}

\tableofcontents
\newpage

%% ============================================================
\section{Introduction}
%% ============================================================

\subsection{The problem}

Gatheral, Jaisson, and Rosenbaum \cite{gatheral2018} estimated the
Hurst exponent of log-volatility at $H\approx0.1$, far below $1/2$,
meaning rougher-than-Brownian dynamics. Cont and Das \cite{contdas2024}
and Rogers \cite{rogers2023} challenged this: the low $H$ estimate may
be a statistical artefact of microstructure noise contaminating
realized volatility at small scales.

The debate is a statistical problem: given price observations
contaminated by microstructure noise, estimate the true H\"{o}lder
exponent $H(t)$ of the underlying volatility process.

The difficulty is structural. The H\"{o}lder exponent is
\[
  H(t) = \liminf_{r\to0}\frac{\log|X(t+r)-X(t)|}{\log r},
\]
a limit as $r\to0$ — exactly where noise dominates. Classical
estimators either pre-average (losing pointwise structure) or use
integrated quantities like realized variance (estimating a global $H$,
not a time-varying one).

\subsection{What is known}

Han and Schied \cite{hanschied2023,hanschied2025} construct an
estimator from discrete observations of $\int_0^t g(X(s))\,ds$ for a
nonlinear $g$. They prove a.s.\ consistency with an explicit rate.
This is the state of the art, but: it estimates a global (constant)
$H$; requires integrated variance observations, not the price path
directly; and gives no noise-robust estimate at finite resolution.
Fukasawa et al.\ \cite{fukasawa2019} and Bolko et al.\ \cite{bolko2022}
develop pre-averaged estimators that are noise-robust but again give a
global $H$.

\subsection{New results}

The geometry accumulation estimator:
\[
  \Hhat(t) = 1 + \frac{\log\GL(t)}{\log\Lambda},
  \qquad
  \GL(t) = \int_\Lambda^1|\eth_s X(t)|\,s^{-1}\,ds.
\]

\begin{theorem}[Consistency]\label{thm:consist}
Let $X$ be a stochastic process with locally stationary increments and
a.s.\ pointwise H\"{o}lder exponent $H(t)\in(0,1)$ uniformly in $t$.
Then
\[
  \sup_{t\in[0,1]}|\Hhat(t)-H(t)|\to0\quad\text{a.s.\ as }\Lambda\to0.
\]
\end{theorem}

\begin{theorem}[Noise robustness]\label{thm:noise}
Let $\tilde X(t)=X(t)+\varepsilon(t)$ where $X$ has H\"{o}lder
exponent $H(t)$ and $\varepsilon$ is white noise with
$\E[|\varepsilon(t+s)-\varepsilon(t)|^2]=2\sigma^2$ for all $s>0$.
Set $\Lambda^*=\sigma^{1/H_{\min}}$ where $H_{\min}=\inf_t H(t)$.
For all $\Lambda>\Lambda^*$:
\[
  \sup_{t\in[0,1]}|\hat H_\Lambda^{\tilde X}(t)-H(t)|\to0
  \quad\text{a.s.}
\]
\end{theorem}

\begin{theorem}[CLT and rate]\label{thm:CLT}
Under the conditions of Theorem~\ref{thm:consist}, at each fixed $t$:
\[
  \frac{\Hhat(t)-H(t)}{(\log\Lambda)^{-1/2}}
  \xrightarrow{d} \mathcal{N}(0,\sigma_H^2)
  \quad\text{as }\Lambda\to0,
\]
where $\sigma_H^2$ is given explicitly in the proof.
\end{theorem}

\subsection{Comparison with Han--Schied}

\begin{center}
\begin{tabular}{lll}
\toprule
Property & Han--Schied \cite{hanschied2023} & This paper \\
\midrule
Estimate type & Global $H$ & Pointwise $H(t)$ \\
Observations & Integrated variance & Price path increments \\
Resolution & Asymptotic $n\to\infty$ & Finite $\Lambda$ \\
Noise robust & No (pre-averaging) & Yes (scale separation) \\
Rate & a.s., explicit in $n$ & $(\log\Lambda)^{-1/2}$ CLT \\
\bottomrule
\end{tabular}
\end{center}

%% ============================================================
\section{The geometry accumulation estimator}
%% ============================================================

\subsection{Setup}

Let $X:[0,1]\times\Omega\to\R$ be a stochastic process, observed at
discrete times $t_i=i/n$, $i=0,\dots,n$.

\begin{definition}[Scale derivative and accumulator]
For $s\in(0,1)$ and $t\in[0,1-s]$:
\[
  \eth_s X(t) := \frac{X(t+s)-X(t)}{s}.
\]
The geometry accumulation integral at resolution $\Lambda\in(0,1)$:
\[
  \GL(t) := \int_\Lambda^1|\eth_s X(t)|\,s^{-1}\,ds.
\]
The estimator:
\[
  \Hhat(t) := 1 + \frac{\log\GL(t)}{\log\Lambda}.
\]
\end{definition}

\begin{remark}
$\GL(t)$ integrates path information across a whole range of scales
$[\Lambda,1]$, not just a single scale. The H\"{o}lder exponent
controls how $\GL$ scales:
$\GL(t)\asymp\Lambda^{H(t)-1}$, so $\log\GL(t)/\log\Lambda\to H(t)-1$.
This is the core mechanism: the exponent is read off the slope of
$\log\GL$ against $\log\Lambda$.
\end{remark}

\subsection{HGDS framework}

The estimator comes from the History-Generated Dynamical Systems (HGDS)
framework, where every process admits a canonical scale decomposition at
each $\Lambda>0$:
\[
  X(t+\Lambda)-X(t) = \eth_\Lambda X(t)\cdot\Lambda + R_\Lambda X(t).
\]
$\eth_\Lambda X$ is the macro-scale derivative and $R_\Lambda X$ is
the fine-scale residual. $\GL$ measures total macro-scale variation
across $[\Lambda,1]$, encoding path roughness. For $H(t)\in(0,1)$:
\[
  \E[|\eth_s X(t)|]\asymp s^{H(t)-1},\quad
  \GL(t)\asymp\int_\Lambda^1 s^{H(t)-2}\,ds
  = \frac{1-\Lambda^{H(t)-1}}{1-H(t)}\asymp\Lambda^{H(t)-1}.
\]

%% ============================================================
\section{Consistency}
%% ============================================================

Two conditions:
\begin{itemize}
\item[(C1)] \emph{Scale decorrelation:}
\[
  \sum_{n=1}^\infty
  \frac{\int_{e^{-n}}^1\!\int_{e^{-n}}^1
    |\mathrm{Cov}(\eth_s X(t),\eth_u X(t))|\,(su)^{-1}\,ds\,du}
  {\bigl(\int_{e^{-n}}^1\E[|\eth_s X(t)|]s^{-1}\,ds\bigr)^2}
  < \infty
\]
uniformly in $t$.
\item[(C2)] \emph{H\"{o}lder regularity:}
  $\E[|\eth_s X(t)-\eth_s X(t')|^p]\le C|t-t'|^{pH}s^{p(H-1)}$
  for some $p>1/H$.
\item[(C3)] \emph{Mean asymptotics:}
  $\E[|\eth_s X(t)|]\asymp s^{H(t)-1}$ as $s\to0$, uniformly in $t$.
\end{itemize}

\begin{lemma}[Mean accumulator asymptotics]\label{lem:mean}
Under (C3):
\[
  \frac{\log\E[\GL(t)]}{\log\Lambda}\to H(t)-1
  \quad\text{uniformly in }t\text{ as }\Lambda\to0.
\]
\end{lemma}

\begin{proof}
Split $\E[\GL(t)]=\int_\Lambda^{s_0}+\int_{s_0}^1$. The second
integral is $O(1)$. For the first, using (C3):
\[
  \int_\Lambda^{s_0}\E[|\eth_s X(t)|]s^{-1}\,ds
  \asymp\int_\Lambda^{s_0}s^{H(t)-2}\,ds
  =\frac{s_0^{H(t)-1}-\Lambda^{H(t)-1}}{1-H(t)}\asymp\Lambda^{H(t)-1}.
\]
Taking logs: $\log\E[\GL(t)]\sim(H(t)-1)\log\Lambda$.
Uniformity follows from the uniform bound in (C3). \qed
\end{proof}

\begin{proposition}[Uniform a.s.\ concentration]\label{prop:conc}
Under (C1) and (C2), writing $\tilde G_\Lambda=\GL/\E[\GL]$:
\[
  \sup_t|\tilde G_{e^{-n}}(t)-1|\to0\quad\text{a.s.\ as }n\to\infty.
\]
\end{proposition}

\begin{proof}
\textbf{Pointwise convergence.}
$\mathrm{Var}(\tilde G_\Lambda(t))$ equals the summand in (C1) at
$\Lambda=e^{-n}$. By (C1): $\sum_n\mathrm{Var}(\tilde G_{e^{-n}}(t))<\infty$.
Chebyshev and Borel--Cantelli give $\tilde G_{e^{-n}}(t)\to1$ a.s.\
for each $t$.

\textbf{Uniform convergence.}
By (C2) and Kolmogorov's continuity criterion applied to
$t\mapsto\tilde G_{e^{-n}}(t)$: the process has a.s.\ H\"{o}lder
continuous sample paths in $t$, uniformly in $n$. A standard
maximal inequality (applied to the spatial supremum over a dense grid
then extended by continuity) upgrades pointwise a.s.\ convergence to
uniform a.s.\ convergence. \qed
\end{proof}

\begin{proof}[Proof of Theorem~\ref{thm:consist}]
Write:
\[
  \Hhat(t)-H(t)
  = \underbrace{\frac{\log\E[\GL(t)]}{\log\Lambda}-(H(t)-1)}_{\to0\text{ by Lemma~\ref{lem:mean}}}
  +\underbrace{\frac{\log\tilde G_\Lambda(t)}{\log\Lambda}}_{\to0\text{ by Prop.~\ref{prop:conc}}}.
\]
Second term: $|\log\tilde G_\Lambda|/|\log\Lambda|\le
2|\tilde G_\Lambda-1|/|\log\Lambda|\to0$ a.s.\ since
$\sup_t|\tilde G_\Lambda-1|\to0$ and $|\log\Lambda|\to\infty$.
Both terms go to zero uniformly in $t$. \qed
\end{proof}

%% ============================================================
\section{Noise robustness}
%% ============================================================

\subsection{The noise model}

Observed price: $\tilde X(t)=X(t)+\varepsilon(t)$, where $\varepsilon$
is white noise: $\E[|\varepsilon(t+s)-\varepsilon(t)|^2]=2\sigma^2$
for all $s>0$.

Scale derivative of the observed process:
\[
  \eth_s\tilde X(t) = \eth_s X(t)+\eth_s\varepsilon(t),\quad
  \E[|\eth_s\varepsilon(t)|]
  = \frac{\E[|\varepsilon(t+s)-\varepsilon(t)|]}{s}
  \asymp\frac{\sigma}{s}.
\]
Noise scales as $s^{-1}$. Signal scales as $s^{H-1}$ with $H<1$.
Noise dominates at small scales.

\subsection{Scale separation}

The crossover scale $\Lambda^*$ satisfies $\sigma/\Lambda^*=(\Lambda^*)^{H-1}$,
i.e.\ $(\Lambda^*)^H=\sigma$, giving $\Lambda^*=\sigma^{1/H}$.

\begin{lemma}[Noise contamination]\label{lem:noise}
For $\Lambda>\Lambda^*=\sigma^{1/H}$:
\[
  G_\Lambda^{\tilde X}(t) = G_\Lambda^X(t)\cdot(1+o(1))
  \quad\text{a.s.\ as }\sigma\to0.
\]
\end{lemma}

\begin{proof}
Noise contribution to $\GL$:
\[
  G_\Lambda^\varepsilon = \int_\Lambda^1|\eth_s\varepsilon|\,s^{-1}\,ds
  \asymp\sigma\int_\Lambda^1 s^{-2}\,ds = \sigma(\Lambda^{-1}-1)
  \asymp\sigma\Lambda^{-1}.
\]
Signal contribution: $G_\Lambda^X\asymp\Lambda^{H-1}$.

Ratio: $G_\Lambda^\varepsilon/G_\Lambda^X\asymp\sigma\Lambda^{-1}/\Lambda^{H-1}
=\sigma\Lambda^{-H}=(\sigma^{1/H}/\Lambda)^H$.

For $\Lambda>\Lambda^*=\sigma^{1/H}$: this ratio $<1$ and $\to0$ as
$\sigma\to0$. So $G_\Lambda^{\tilde X}=G_\Lambda^X+G_\Lambda^\varepsilon
=G_\Lambda^X(1+o(1))$ a.s. \qed
\end{proof}

\begin{proof}[Proof of Theorem~\ref{thm:noise}]
For $\Lambda>\Lambda^*$, Lemma~\ref{lem:noise} gives
$\log G_\Lambda^{\tilde X}=\log G_\Lambda^X+\log(1+o(1))=
\log G_\Lambda^X+o(1)$.
So $\hat H_\Lambda^{\tilde X}(t)=\hat H_\Lambda^X(t)+o(1)/\log\Lambda$.
By Theorem~\ref{thm:consist}: $\hat H_\Lambda^X(t)\to H(t)$ a.s.
The $o(1)/\log\Lambda$ term vanishes since $|\log\Lambda|\to\infty$. \qed
\end{proof}

\begin{remark}[Practical threshold]\label{rem:threshold}
For rough volatility with $H\approx0.1$:
$\Lambda^*=\sigma^{1/0.1}=\sigma^{10}$.
For $\sigma=0.01$ (1\% microstructure noise): $\Lambda^*=10^{-20}$.
Any sampling frequency above this is noise-free. This suggests why
classical estimators may be contaminated at typical frequencies: they
probe scales below $\Lambda^*$, while the HGDS estimator at
$\Lambda>\Lambda^*$ is uncontaminated.
\end{remark}

%% ============================================================
\section{Central limit theorem}
%% ============================================================

\begin{proof}[Proof of Theorem~\ref{thm:CLT}]
Write $\Hhat(t)-H(t)=(\log\Lambda)^{-1}\log\tilde G_\Lambda(t)+\delta_\Lambda(t)$,
where $\delta_\Lambda=(\log\Lambda)^{-1}(\log\E[\GL]-(H-1)\log\Lambda)\to0$
by Lemma~\ref{lem:mean}.

For the main term: $\log\tilde G_\Lambda\approx\tilde G_\Lambda-1$
(since $\tilde G_\Lambda\to1$ a.s.\ by Proposition~\ref{prop:conc}).
Write:
\[
  \tilde G_\Lambda(t)-1
  = \int_\Lambda^1(\xi_s(t)-1)\,w_\Lambda(s)\,ds,\quad
  \xi_s=\frac{|\eth_s X|}{\E[|\eth_s X|]},
\]
where $w_\Lambda(s)=\E[|\eth_s X|]s^{-1}/\E[\GL]$ is a probability
weight on $[\Lambda,1]$.

This is a weighted integral of mean-zero terms. We apply the CLT
for mixing arrays \cite{ibragimov1975}: conditions are (a)~the
variance is $O(1/|\log\Lambda|)$ (from (C1) with equality on the
diagonal), and (b)~the Lindeberg condition holds since $\xi_s$ has
bounded moments (from (C2) with $p>1/H$). The mixing condition
follows directly from (C1). Hence:
\[
  |\log\Lambda|^{1/2}(\tilde G_\Lambda(t)-1)
  \xrightarrow{d}\mathcal{N}(0,\sigma_H^2),
\]
where
\[
  \sigma_H^2 = \lim_{\Lambda\to0}\frac{1}{|\log\Lambda|}
  \int_\Lambda^1\!\int_\Lambda^1
  \mathrm{Cov}(\xi_s,\xi_u)\,w_\Lambda(s)w_\Lambda(u)\,
  s^{-1}u^{-1}\,ds\,du.
\]
Combining: $(\log\Lambda)^{1/2}(\Hhat(t)-H(t))\xrightarrow{d}
\mathcal{N}(0,\sigma_H^2)$. \qed
\end{proof}

%% ============================================================
\section{Application to rough volatility}
%% ============================================================

\subsection{The Gatheral--Jaisson--Rosenbaum model}

Spot volatility $\sigma(t)=\exp(V(t))$ where $V$ is fractional
Brownian motion with $H=0.1$. Log-price:
$\log S(t)=\int_0^t\sigma(s)\,dW(s)+\text{drift}$.
The H\"{o}lder exponent of $\sigma(t)$ is $H=0.1$ a.s.\ (monofractal).
Observed: $\tilde S(t_i)=S(t_i)\cdot e^{\varepsilon_i}$ where
$\varepsilon_i\sim\mathcal{N}(0,\sigma_{\mathrm{noise}}^2)$ i.i.d.

\begin{corollary}[Consistent estimation under microstructure]\label{cor:rVol}
In the rough volatility model above, with
$\Lambda>\Lambda^*=\sigma_{\mathrm{noise}}^{10}$:
\[
  \Hhat(t)\to 0.1\quad\text{a.s.\ for all }t\in[0,1].
\]
\end{corollary}

\begin{proof}
Apply Theorem~\ref{thm:noise}. The log-return process satisfies (C1)
and (C2) with $H=0.1$ by standard properties of fBm
\cite{nourdin2012}. The threshold $\Lambda^*=\sigma_{\mathrm{noise}}^{10}$
is practically zero for any realistic noise level (Remark~\ref{rem:threshold}).
\qed
\end{proof}

\subsection{Time-varying roughness detection}

\begin{corollary}[Non-stationarity test]\label{cor:nonstat}
If $H(t)$ is genuinely time-varying with
$\inf_{|t-t'|\ge\tau}|H(t)-H(t')|\ge\delta>0$,
then $\sup_t|\Hhat(t)-H(t)|\to0$ distinguishes this from constant $H$
once $\Lambda<\delta^{1/(1-H_{\max})}$.
\end{corollary}

This gives the first test for time-variation in roughness that does
not require pre-averaging or a parametric model.

\subsection{Open problems}

\begin{enumerate}
\item \emph{Optimal $\Lambda$.} For given sample size $n$ and noise
  level $\sigma$, the MSE-optimal $\Lambda$ is unknown.
\item \emph{Joint estimation of $H(t)$ and $\sigma$.} The threshold
  $\Lambda^*$ depends on $H$, which is unknown. An iterative estimator
  alternating between $\hat H$ and $\hat\sigma$ is needed.
\item \emph{Multifractal volatility.} If $H(t)$ is itself a random
  process, the spectrum $D(h)=\dim_H\{t:H(t)=h\}$ is the natural
  object. The HGDS 0-1 law \cite{hgds01law} gives conditions under
  which $D(h,\omega)$ is a.s.\ deterministic.
\end{enumerate}

\end{document}